\documentclass[12pt,twoside]{amsart}
\usepackage[all]{xy}
\usepackage{graphicx}

\title[
Terminal toric Fano three-folds with numerical conditions]
{
Terminal toric Fano three-folds with certain numerical conditions} 
\author{Hiroshi Sato and Ryota Sumiyoshi} 
\date{2019/10/30, version 0.12}
\address{Department of Applied Mathematics, Faculty of Sciences, 
Fukuoka University, 8-19-1, Nanakuma, Jonan-ku, Fukuoka 814-0180, Japan}
\email{hirosato@fukuoka-u.ac.jp}
\address{Department of Applied Mathematics, Faculty of Sciences, 
Fukuoka University, 8-19-1, Nanakuma, Jonan-ku, Fukuoka 814-0180, Japan}
\email{sd170002@cis.fukuoka-u.ac.jp}

\makeatletter
    
    \@addtoreset{equation}{section}
\makeatother

\newcommand{\mult}[0]{{\operatorname{mult}}}

\newcommand{\N}[0]{{\operatorname{N}}}
\newcommand{\G}[0]{{\operatorname{G}}}
\newcommand{\Z}[0]{{\operatorname{Z}}}

\newtheorem{thm}{Theorem}[section]
\newtheorem{lem}[thm]{Lemma}
\newtheorem{cor}[thm]{Corollary}
\newtheorem{prop}[thm]{Proposition}

\theoremstyle{definition}
\newtheorem{ex}[thm]{Example}
\newtheorem{que}[thm]{Question}
\newtheorem{defn}[thm]{Definition}
\newtheorem{rem}[thm]{Remark}
\newtheorem*{ack}{Acknowledgments}       

\begin{document}
\bibliographystyle{amsalpha+}

\begin{abstract}
We completely classify the $\mathbb{Q}$-factorial 
terminal toric Fano three-folds such that 
the sum of the squared torus invariant prime divisors 
is non-negative.
\end{abstract}

\maketitle

\tableofcontents
\section{Introduction} %%%%%%%%%%%%%%%%%%%

A smooth projective variety $X$ is a {\em Fano} manifold 
if its first Chern character 
${\rm ch}_1(X)={\rm c}_1(X)=-K_X$ is an ample divisor. 
Although the definition of Fano manifolds is very simple, 
there are a lot of important properties for them. 

In order to investigate the existence of rational surfaces 
on a Fano manifold, the notion of 
$2$-Fano manifolds was introduced in \cite{starr}:

\begin{defn}
A Fano manifold $X$ is a $2$-{\em Fano} manifold if the 
second Chern character ${\rm ch}_2(X)$ is nef, that is, 
${\rm ch}_2(X)\cdot S\ge 0$ for any surface $S$ on $X$. 
\end{defn}

For recent results for the classification of 
$2$-Fano manifolds, see \cite{castravet}. 
Few examples of $2$-Fano manifolds are known so far. 
This means that 
the condition that ${\rm ch}_2(X)$ is nef strongly 
determines the structure of a Fano manifold. 

For the case of toric manifolds, 
we can see this phenomenon without 
the condition that $X$ is Fano. Namely, 
it seems that projective toric manifolds whose 
second Chern characters are nef have common geometric 
properties (see \cite{nobili}, \cite{sato1} and \cite{sato2}). 

In this paper, 
we  generalize these researches for singular cases. 
Let $X$ be a $\mathbb{Q}$-factorial projective toric 
$d$-fold and $D_1,\ldots,D_n$ the torus invariant 
prime divisors on $X$. Put
\[
\gamma_2(X):=D_1^2+\cdots+D_n^2.
\]
If $X$ is smooth, it is well known that 
$2{\rm ch}_2(X)=\gamma_2(X)$. 
So, we investigate a singular toric variety $X$ 
such that $\gamma_2(X)$ is 
positive or 
nef (see Definition \ref{gammapositive}). 
In particular, we completely determine which 
$\mathbb{Q}$-factorial terminal toric Fano $3$-fold 
is $\gamma_2$-nef. 
There exist exactly 
$23$ $\mathbb{Q}$-factorial terminal 
$\gamma_2$-nef toric 
Fano $3$-folds (see Theorem \ref{mfo}). 
To prove the main theorem, 
we introduce the useful formulas for the 
calculation of intersection numbers. 

This paper is organized as follows: 
In Section \ref{prelimi}, we collect basic results for 
toric varieties. In Section \ref{forintersection}, 
we introduce the notion of $\gamma_2$-positive 
and $\gamma_2$-nef toric varieties 
which are main objects of this paper. 
Moreover, we introduce the formulas for 
the calculation of intersection numbers which 
are useful for checking 
whether a given toric variety is 
$\gamma_2$-nef or not. 
In Section \ref{delpezzo}, we determine 
which Gorenstein toric del Pezzo surface is $\gamma_2$-nef. 
This classification is useful to understand the results of 
Section \ref{fano3}. In Section \ref{fano3}, we complete 
the classification of $\mathbb{Q}$-factorial terminal 
$\gamma_2$-nef toric Fano $3$-folds. 
One can see that they have some common geometric properties. 

\begin{ack}
The authors would like to thank Professor 
Akihiro Higashitani for teaching them about the software 
\texttt{polymake} 
which is useful for computing convex lattice polytopes. 
They also thank Professors 
Yasuhiro Nakagawa and Takahiko Yoshida 
for comments. 
The first author was partially supported by JSPS KAKENHI 
Grant Number JP18K03262. 
Finally, the authors would like to thank the referee for giving them some polite and 
concrete comments. 
\end{ack}

%%%%%%%%%%%%%%%%%%%%%%%%%%%%%%%
\section{Preliminaries}\label{prelimi} %%%%%%%%%%%%%%%%%%%

In this section, we introduce some basic results and notation of 
toric varieties and intersection numbers on them. 
For the details, please see \cite{cls}, 
\cite{fulton} and  \cite{oda}.
%See also \cite{fujino-sato}, \cite[Chapter 14]{matsuki} 
%and \cite{reid} for the toric Mori theory. 

Let $X=X_\Sigma$ be the toric $d$-fold associated to a fan 
$\Sigma$ in $N=\mathbb{Z}^d$ 
over an algebraically closed field $k$ of arbitrary characteristic. 
Put $N_{\mathbb{R}}:=N\otimes\mathbb{R}$. 
It is well known that there exists a one-to-one correspondence between 
the $r$-dimensional cones in $\Sigma$ and the torus invariant 
subvarieties of dimension $d-r$ in $X$. Let $\G(\Sigma)$ 
be the set of primitive generators for $1$-dimensional cones in $\Sigma$. 
Thus, for $v\in\G(\Sigma)$, 
we have the torus invariant prime divisor corresponding to 
$\mathbb{R}_{\ge 0}v\in\Sigma$.

For an $r$-dimensional simplicial cone $\sigma\in\Sigma$, 
let $N_\sigma\subset N$ be the sublattice generated by $\sigma\cap N$ 
and let $\sigma\cap\G(\Sigma)=\{v_1,\ldots,v_r\}$. 
Put 
\[
\mult(\sigma):=[N_\sigma:\mathbb{Z}v_1+\cdots+\mathbb{Z}v_r]
\]
which is the index of the subgroup $\mathbb{Z}v_1+\cdots+\mathbb{Z}v_r$ 
in $N_\sigma$. 
The following property for intersection numbers 
on a toric variety is fundamental. 

\begin{prop}\label{kotensu}
Let $X$ be a $\mathbb{Q}$-factorial toric $d$-fold, and let 
$\sigma,\tau\in\Sigma$, $V_{\sigma}$ and $V_{\tau}$ the 
torus invariant subvarieties associated to $\sigma$ and $\tau$. 
If $\sigma$ and $\tau$ span $\lambda\in \Sigma$ with 
$\dim \lambda=\dim \sigma +\dim \tau$, 
then 
\[
V_{\sigma}\cdot V_{\tau}=\frac{\mult (\sigma)\cdot \mult(\tau)}
{\mult (\lambda)} V_{\lambda},
\]
where $V_{\lambda}$ is the torus invariant subvariety 
associated to $\lambda$. 
On the other hand, 
if $\sigma$ and $\tau$ are contained in no cone of 
$\Sigma$, then $V_{\sigma}\cdot V_{\tau}=0$. 
\end{prop}

Let $X$ be a projective toric $d$-fold. 
For $1\le i\le d$, 
we put 
\[
\Z_i(X):=\{\mbox{the }i\mbox{-cycles on } X\},
\mbox{ while }
\Z^i(X):=\{\mbox{the }i\mbox{-cocycles on } X\}.
\]
We introduce the numerical equivalence $\equiv$ on $\Z_i(X)$ 
and $\Z^i(X)$ as follows: For $C\in\Z_i(X)$, we define 
$C\equiv 0$ if $D\cdot C=0$ for any $D\in \Z^i(X)$, while 
for $D\in\Z^i(X)$, we define 
$D\equiv 0$ if $D\cdot C=0$ for any $C\in \Z_i(X)$. 
We put 
\[
\N_i(X):=\left(\Z_i(X)\otimes{\mathbb R}\right)/\equiv, 
\mbox{ while }
\N^i(X):=\left(\Z^i(X)\otimes \mathbb R\right)/\equiv.
\]

\section{Formulas for intersection numbers}\label{forintersection}
 %%%%%%%%%%%

In this section, we introduce the notion of $\gamma_r$-positive 
toric varieties, which are main objects of this paper. 
Also, we prepare the formulas for intersection numbers which 
are useful for  our calculations in Section \ref{fano3}. 

\begin{defn}\label{gammapositive}
Let $X$ be a $\mathbb{Q}$-factorial projective toric 
$d$-fold. For any integer $1\le r\le d$, put 
\[
\gamma_r=\gamma_r(X):=D_1^r+\cdots+D_n^r\in\N^r(X),
\]
where $D_1,\ldots,D_n$ are the torus invariant prime divisors. 

If $\gamma_r\cdot Y>0$ (resp. $\ge 0$) for any $r$-dimensional 
torus invariant subvariety $Y\subset X$, 
then we say that $X$ is $\gamma_r$-{\em positive}  
(resp. $\gamma_r$-{\em nef}).   
\end{defn}
\begin{rem}
For the case where $X$ is smooth, it is well known that  
\[
\frac{1}{r!}\gamma_r(X)={\rm ch}_r(X)
\]
which is the $r$-th Chern character.
\end{rem}
The main purpose of this paper is to determine 
which $\mathbb{Q}$-factorial terminal toric Fano $3$-fold is 
$\gamma_2$-positive or $\gamma_2$-nef  
(see Section \ref{fano3}). 

For the classification, we introduce the following notion 
about intersection numbers on $X$. This is convenient 
for the calculation of intersection numbers.

\begin{defn}\label{polynomial}
Let $X=X_\Sigma$ be a $\mathbb{Q}$-factorial projective toric 
$d$-fold, and $S\subset X$ a torus invariant subsurface 
on $X$. Let $\G(\Sigma)=\{x_1,\ldots,x_n\}$, and we consider 
the polynomial ring $\mathcal{R}(X):=
\mathbb{Q}[X_1,\ldots,X_n]$, where 
$X_1,\ldots,X_n$ are the 
independent variables of polynomials corresponding 
to $x_1,\ldots,x_n$, respectively. Let $D_1,\ldots,D_n$ 
be the torus invariant prime divisors corresponding to 
$x_1,\ldots,x_n$, respectively. Then, we define 
the quadric homogeneous 
polynomial $I_{S/X}\in \mathbb{Q}[X_1,\ldots,X_n]$ 
as follows:
\[
I_{S/X}=I_{S/X}(X_1,\ldots,X_n):=
\sum_{1\le i,j\le n}(D_i\cdot D_j\cdot S)X_iX_j.
\]
\end{defn}

For the cases where $\rho(S)=1$ or $2$, 
we can determine $I_{S/X}$ explicitly as follows. 

\begin{prop}\label{rho1lemma} %%%
Let $X=X_\Sigma$ be a $\mathbb{Q}$-factorial projective toric 
$d$-fold, and $S\subset X$ a torus invariant subsurface 
on $X$. 
If $\rho(S)=1$, then $\gamma_2\cdot S>0$. 
\end{prop}
\begin{proof}
Let $\tau\in\Sigma$ be a $(d-2)$-dimensional cone 
associated to $S$ and 
$\tau\cap\G(\Sigma)=\{x_1,\ldots,x_{d-2}\}$. 
There exist exactly $3$ maximal cones 
\[
\mathbb{R}_{\ge 0}x_{d-1}+\mathbb{R}_{\ge 0}x_d+\tau,\ 
\mathbb{R}_{\ge 0}x_d+\mathbb{R}_{\ge 0}x_{d+1}+\tau,\ 
\mathbb{R}_{\ge 0}x_{d+1}+\mathbb{R}_{\ge 0}x_{d-1}\tau
\] 
in $\Sigma$, where $\{x_{d-1},x_d,x_{d+1}\}\subset\G(\Sigma)$. 
There exists a linear relation 
\[
a_1x_1+\cdots+a_{d-2}x_{d-2}+
a_{d-1}x_{d-1}+a_dx_d+a_{d+1}x_{d+1}=0,
\]
where $a_1,\ldots,a_{d+1}\in\mathbb{Q}$ and 
$a_{d-1},a_d,a_{d+1}>0$. We remark that this relation is 
unique up to multiplication by a {\em positive} rational number. 
Since $\{x_1,\ldots,x_d\}$ is an 
$\mathbb{R}$-basis for $N_{\mathbb{R}}$, 
we have $d$ relations 
\[
D_i-\frac{a_i}{a_{d+1}}D_{d+1}+F_i=0\mbox{ for }
1\le i\le d,
\]
in $\N^1(X)$, where $D_1,\ldots,D_{d+1}$ are 
torus invariant prime divisors corresponding to 
$x_1,\ldots,x_{d+1}$, respectively, while $F_1,\ldots,F_d$ 
are torus invariant divisors which do not intersect $S$. 
So, we can ignore $F_i$'s in the following calculation. 
For any $1\le i,j\le d$, 
\[
D_i\cdot D_j\cdot S=
\left(\frac{a_i}{a_{d+1}}\times\frac{a_{d+1}}{a_{d-1}}D_{d-1}\right)\cdot
\left(\frac{a_j}{a_{d+1}}\times\frac{a_{d+1}}{a_d}D_d\right)\cdot S=
\frac{a_ia_j}{a_{d-1}a_{d}}D_{d-1}D_d\cdot S.
\]
This means that 
there exists a rational positive number $\alpha$ such that 
for any $1\le i,j\le d+1$, we have 
$\alpha D_i\cdot D_j\cdot S=a_ia_j$. 
In particular, we have 
\[
\gamma_2\cdot S=\frac{a_1^2+\cdots+a_{d+1}^2}{\alpha}>0.
\] 
\end{proof}

\begin{rem}
The calculation in the proof of Proposition \ref{rho1lemma} 
tells us that 
\[
\alpha I_{S/X}=(a_1X_1+\cdots+a_{d+1}X_{d+1})^2,
\]
where $X_1,\ldots,X_{d+1}$ are the 
independent variables of polynomials corresponding 
to $x_1,\ldots,x_{d+1}$, respectively. 
\end{rem}

Although the case where $\rho(S)=2$ is more complicated, 
we can determine $I_{S/X}$ as follows:

Let $X=X_\Sigma$ be a $\mathbb{Q}$-factorial projective toric 
$d$-fold, and $S\subset X$ a torus invariant subsurface 
of $\rho(S)=2$. 
Let $\tau\in\Sigma$ be a $(d-2)$-dimensional cone 
associated to $S$ and 
$\tau\cap\G(\Sigma)=\{x_1,\ldots,x_{d-2}\}$. Then, 
there exist exactly $4$ maximal cones 
\[
\mathbb{R}_{\ge 0}y_1+\mathbb{R}_{\ge 0}y_3+\tau,\ 
\mathbb{R}_{\ge 0}y_2+\mathbb{R}_{\ge 0}y_3+\tau,\ 
\mathbb{R}_{\ge 0}y_1+\mathbb{R}_{\ge 0}y_4+\tau,\ 
\mathbb{R}_{\ge 0}y_2+\mathbb{R}_{\ge 0}y_4+\tau 
\] 
in $\Sigma$, where $\{y_1,y_2,y_3,y_4\}\subset\G(\Sigma)$. 
For $(d-1)$-dimensional cones 
$\mathbb{R}_{\ge 0}y_3+\tau$ and  
$\mathbb{R}_{\ge 0}y_1+\tau$, we have linear relations 
\[
b_1y_1+b_2y_2+c_3y_3+a_1x_1+\cdots+a_{d-2}x_{d-2}=0
\mbox{ and }
\]
\[
b_3y_3+b_4y_4+c_1y_1+e_1x_1+\cdots+e_{d-2}x_{d-2}=0,
\]
respectively, where $a_1,\ldots,a_{d-2},b_1,b_2,b_3,b_4,
c_1,c_3,e_1,\ldots,e_{d-2}\in\mathbb{Q}$ and 
$b_1,b_2,b_3,b_4>0$. Then, the following holds.

\begin{prop}\label{rho2lemma} %%%%%%%%%%%%%%%%
Under the above setting, we have
\[
\alpha I_{S/X}=
-b_3c_1\left(b_1Y_1+b_2Y_2+c_3Y_3+\sum_{i=1}^{d-2}a_iX_i\right)^2+
\]
\[
2b_1b_3
\left(b_1Y_1+b_2Y_2+c_3Y_3+\sum_{i=1}^{d-2}a_iX_i\right)
\left(b_3Y_3+b_4Y_4+c_1Y_1+\sum_{i=1}^{d-2}e_iX_i\right)
\]
\[
-b_1c_3\left(b_3Y_3+b_4Y_4+c_1Y_1+\sum_{i=1}^{d-2}e_iX_i\right)^2,
\]
for some positive rational number $\alpha$, where  
$X_1,\ldots,X_{d-2},Y_1,Y_2,Y_3,Y_4$ are the 
independent variables of polynomials corresponding 
to $x_1,\ldots,x_{d-2},y_1,y_2,y_3,y_4$, respectively. 
In particular,  
\[
\alpha \gamma_2\cdot S=
-b_3c_1\left(b_1^2+b_2^2+c_3^2+\sum_{i=1}^{d-2}a_i^2\right)
+2b_1b_3\left(b_1c_1+b_3c_3+\sum_{i=1}^{d-2}a_ie_i\right)
\]
\[
-b_1c_3\left(b_3^2+b_4^2+c_1^2+\sum_{i=1}^{d-2}e_i^2\right).
\]
\end{prop}

\begin{proof}
$\{x_1,\ldots,x_{d-2},y_1,y_3\}$ is an $\mathbb{R}$-basis 
for $N_\mathbb{R}$, and by using this basis, 
$y_2$ and $y_4$ are expressed as 
\[
y_2=-\frac{1}{b_2}\left(
b_1y_1+c_3y_3+\sum_{i=1}^{d-2}a_ix_i\right)\mbox{ and }
y_4=-\frac{1}{b_4}\left(
b_3y_3+c_1y_1+\sum_{i=1}^{d-2}e_ix_i\right),
\]
respectively. Thus, we obtain $d$ relations
\[
\left(D_i-\frac{a_i}{b_2}E_2-\frac{e_i}{b_4}E_4\right)\cdot S=0
\mbox{ for }1\le i\le d-2, 
\]
\[
(1)\ \left(E_1-\frac{b_1}{b_2}E_2-\frac{c_1}{b_4}E_4\right)
\cdot S=0\mbox{ and } 
(2)\ \left(E_3-\frac{c_3}{b_2}E_2-\frac{b_3}{b_4}E_4\right)
\cdot S=0
\]
in $\N_1(X)$, where $D_1,\ldots,D_{d-2},E_1,E_2,E_3,E_4$ 
are the torus invariant prime divisors corresponding to 
$x_1,\ldots,x_{d-2},y_1,y_2,y_3,y_4$, respectively. 
On the other hand, $E_1E_2\cdot S=E_3E_4\cdot S=0$ and 
$E_1E_3\cdot S,\ E_2E_4\cdot S>0$. So, 
we can express the intersection numbers by $E_2E_4\cdot S$. 

By multiplying the equality $(1)$ by $E_1,E_2,E_3,E_4$, 
we obtain 
\[
E_1^2\cdot S=\frac{c_1}{b_4}E_1E_4\cdot S,\ 
E_2^2\cdot S=-\frac{b_2c_1}{b_1b_4}E_2E_4\cdot S,
\]
\[
E_1E_3\cdot S=\frac{b_1}{b_2}E_2E_3\cdot S,\ 
E_1E_4\cdot S=\frac{b_1}{b_2}E_2E_4\cdot S+
\frac{c_1}{b_4}E_4^2\cdot S, 
\]
while by multiplying the equality $(2)$ by $E_1,E_2,E_3,E_4$, 
we obtain 
\[
E_1E_3\cdot S=\frac{b_3}{b_4}E_1E_4\cdot S,\ 
E_2E_3\cdot S=\frac{c_3}{b_2}E_2^2\cdot S+
\frac{b_3}{b_4}E_2E_4\cdot S,\ 
\]
\[
E_3^2\cdot S=\frac{c_3}{b_2}E_2E_3\cdot S,\ 
E_4^2\cdot S=-\frac{b_4c_3}{b_2b_3}E_2E_4\cdot S.
\]
Thus, we have the equalities 
\[
E_1E_4\cdot S=\left(\frac{b_1}{b_2}-\frac{c_1c_3}{b_2b_3}\right)
E_2E_4\cdot S,\ 
E_1^2\cdot S=\frac{c_1}{b_4}
\left(\frac{b_1}{b_2}-\frac{c_1c_3}{b_2b_3}\right)
E_2E_4\cdot S,
\]
\[
E_1E_3\cdot S=\frac{b_3}{b_4}
\left(\frac{b_1}{b_2}-\frac{c_1c_3}{b_2b_3}\right)
E_2E_4\cdot S,\ 
E_2E_3\cdot S=\left(\frac{b_3}{b_4}-\frac{c_1c_3}{b_1b_4}\right)
E_2E_4\cdot S,
\]
\[
E_3^2\cdot S=\frac{c_3}{b_2}
\left(\frac{b_3}{b_4}-\frac{c_1c_3}{b_1b_4}\right)
E_2E_4\cdot S.
\]
On the other hand, for any $1\le i,j\le d-2$, we have 
\[
D_iD_j\cdot S=\left(\frac{a_i}{b_2}E_2+\frac{e_i}{b_4}E_4\right)
\cdot
\left(\frac{a_j}{b_2}E_2+\frac{e_j}{b_4}E_4\right)\cdot S
\]
\[
=\frac{a_ia_j}{b_2^2}E_2^2\cdot S+\frac{e_ie_j}{b_4^2}E_4^2\cdot S
+\frac{a_ie_j+a_je_i}{b_2b_4}E_2E_4\cdot S
\]
\[
=\left(-\frac{a_ia_jc_1}{b_1b_2b_4}-
\frac{c_3e_ie_j}{b_2b_3b_4}+\frac{a_ie_j+a_je_i}{b_2b_4}
\right)E_2E_4\cdot S,
\]
while 
\[
D_iE_1\cdot S=\frac{e_i}{b_4}E_1E_4\cdot S=
\frac{e_i}{b_4}\left(
\frac{b_1}{b_2}-\frac{c_1c_3}{b_2b_3}
\right)E_2E_4\cdot S,
\]
\[
D_iE_2\cdot S=\frac{a_i}{b_2}E_2^2\cdot S
+\frac{e_i}{b_4}E_2E_4\cdot S=
\left(
\frac{e_i}{b_4}-\frac{a_ic_1}{b_1b_4}
\right)E_2E_4\cdot S,
\]
\[
D_iE_3\cdot S=\frac{a_i}{b_2}E_2E_3\cdot S=
\frac{a_i}{b_2}\left(
\frac{b_3}{b_4}-\frac{c_1c_3}{b_1b_4}
\right)E_2E_4\cdot S,
\]
\[
D_iE_4\cdot S=\frac{a_i}{b_2}E_2E_4\cdot S
+\frac{e_i}{b_4}E_4^2\cdot S=
\left(
\frac{a_i}{b_2}-\frac{c_3e_i}{b_2b_3}
\right)E_2E_4\cdot S,
\]
for $1\le i\le d-2$. 
Put $\beta:=D_2D_4\cdot S$. 
By multiplying these intersection numbers by 
the positive rational number $b_1b_2b_3b_4$, 
we have the following tables of intersection numbers:
\begin{table}[htb]
  \begin{tabular}{|c||c|c|c|c|} \hline
$\displaystyle{\frac{b_1b_2b_3b_4}{\beta}E_iE_j\cdot S}$
& $E_1$ & $E_2$ & $E_3$ & $E_4$ \\ \hline \hline
    $E_1$ & $b_1^2b_3c_1-b_1c_1^2c_3$ & $0$ & 
    $b_1^2b_3^2-b_1b_3c_1c_3$ & 
    $b_1^2b_3b_4-b_1b_4c_1c_3$ \\ \hline
    $E_2$ 
%& $0$
& & $-b_2^2b_3c_1$ & 
    $b_1b_2b_3^2-b_2b_3c_1c_3$ & $b_1b_2b_3b_4$ \\ \hline
    $E_3$ 
%& $b_1^2b_3^2-b_1b_3c_1c_3$ &
%   $b_1b_2b_3^2-b_2b_3c_1c_3$
 & & &
    $b_1b_3^2c_3-b_3c_1c_3^2$ & $0$ \\ \hline
    $E_4$ 
%& $b_1^2b_3b_4-b_1b_4c_1c_3$ & 
  %  $b_1b_2b_3b_4$ & $0$ & 
& & & &
$-b_1b_4^2c_3$ \\ \hline
  \end{tabular}
\end{table}

\begin{table}[htb]
  \begin{tabular}{|c||c|c|c|c|} \hline
$\displaystyle{\frac{b_1b_2b_3b_4}{\beta}D_iE_j\cdot S}$
& $E_1$ & $E_2$ & $E_3$ & $E_4$ \\ \hline \hline
    $D_i$ & $b_1^2b_3e_i-b_1c_1c_3e_i$ & 
$b_1b_2b_3e_i-a_ib_2b_3c_1$ & 
    $a_ib_1b_3^2-a_ib_3c_1c_3$ & 
    $a_ib_1b_3b_4-b_1b_4c_3e_i$ \\ \hline
  \end{tabular}
\end{table}
\noindent
for $1\le i\le d-2$. 
\[
\frac{b_1b_2b_3b_4}{\beta}D_iD_j\cdot S=
-a_ia_jb_3c_1-e_ie_jb_1c_3+a_ie_jb_1b_3+a_je_ib_1b_3\mbox{ for }
1\le i,j\le d-2.
\]
Put 
\[
f_1:=b_1Y_1+b_2Y_2+c_3Y_3,\ f_2:=b_3Y_3+b_4Y_4+c_1Y_1,
\]
\[
g_1:=\sum_{i=1}^{d-2}a_iX_i
\mbox{ and } 
g_2:=\sum_{i=1}^{d-2}e_iX_i
\in\mathcal{R}(X).
\]
Then, 
\[
-b_3c_1\left(b_1Y_1+b_2Y_2+c_3Y_3+
\sum_{i=1}^{d-2}a_iX_i\right)^2+
\]
\[
2b_1b_3
\left(b_1Y_1+b_2Y_2+c_3Y_3+\sum_{i=1}^{d-2}a_iX_i\right)
\left(b_3Y_3+b_4Y_4+c_1Y_1+
\sum_{i=1}^{d-2}e_iX_i\right)
\]
\[
-b_1c_3\left(b_3Y_3+b_4Y_4+c_1Y_1+\sum_{i=1}^{d-2}e_iX_i\right)^2
\]
\[
=-b_3c_1(f_1+g_1)^2+
2b_1b_3
(f_1+g_1)
(f_2+g_2)
-b_1c_3(f_2+g_2)^2
\]
\[
=\left(-b_3c_1f_1^2+2b_1b_3f_1f_2-b_1c_3f_2^2\right)+
\left(-b_3c_1g_1^2+2b_1b_3g_1g_2-b_1c_3g_2^2\right)
\]
\[
+\left(-2b_3c_1f_1g_1+2b_1b_3f_1g_2+2b_1b_3f_2g_1
-2b_1c_3f_2g_2\right).
\]
Thus, one can easily check that every coefficient of 
this polynomial coincides with the intersection number 
calculated above.
\end{proof}

In Proposition \ref{rho2lemma}, 
if $c_1=c_3=a_1e_1=\cdots =a_{d-2}e_{d-2}=0$, 
we have $\gamma_2\cdot S=0$. In particular, the following holds:

\begin{prop}\label{directproduct}
Let $X$ be a $\mathbb{Q}$-factorial projective toric $d$-fold. 
If there exists a toric finite morphism 
$\pi:X'\to X$ such that $X'$ is a direct product of lower-dimensional  
$\mathbb{Q}$-factorial projective 
$\gamma_2$-nef toric varieties, 
then $X$ is also $\gamma_2$-nef (but not $\gamma_2$-positive). 
\end{prop}
\begin{proof}
It is sufficient to prove for the case where $X$ itself is a direct product. 
Suppose $X=X_1\times X_2$. For any torus invariant curves 
$C_1\subset X_1$ and $C_2\subset X_2$, put $S:=C_1\times C_2$. Then, 
for this $S$, the set of elements in $\G(\Sigma)$ which appear in 
the left-hand side of 
\[
b_1y_1+b_2y_2+c_3y_3+a_1x_1+\cdots+a_{d-2}x_{d-2}=0
\]
and the one of 
\[
b_3y_3+b_4y_4+c_1y_1+e_1x_1+\cdots+e_{d-2}x_{d-2}=0
\]
have no common element, 
because $X$ is a direct product. 
Since $b_1\neq 0$ and $b_3\neq 0$,  consequently,  $c_1=c_3=a_1e_1=\cdots =a_{d-2}e_{d-2}=0$. 
Thus, 
we have 
$\gamma_2\cdot (C_1\times C_2)=0$ as seen above. 
Therefore, $X$ is $\gamma_2$-nef (but not $\gamma_2$-positive). 
\end{proof}

%%%%%%%%%%%%%%%%%%%%%%%%%
%%%%%%%%%%%%%%%%%%%%%%%%%%%%%%%%%%
\section{Gorenstein del pezzo surfaces}\label{delpezzo} %%%%%%
%%%%%%%%%%%%%%%%%%%%%%%%%%%%%%%%%%
%%%%%%%%%%%%%%%%%%%%%55

In order to understand the $3$-folds studied in Section \ref{fano3} 
more deeply, 
we classify the Gorenstein $\gamma_2$-nef toric del Pezzo surfaces 
in this section. 
In this case, $\gamma_2$ is a rational number. 

\smallskip

In general, the case where $\rho(X)=1$ for any dimension 
can be determined easily as follows:

\begin{prop}\label{rho1general}
If $X=X_\Sigma$ is a $\mathbb{Q}$-factorial 
projective toric $d$-fold of $\rho(X)=1$, then 
$X$ is $\gamma_2$-positive.
\end{prop}
\begin{proof}
Any torus invariant subsurface of $X$ is of Picard number $1$. 
Therefore, this proposition is an immediate consequence from Proposition 
\ref{rho1lemma}. 
\end{proof}

For a projective toric surface $S$, the polynomial $I_{S/S}$ defined in 
Definition \ref{polynomial} can be calculated inductively: 

Let $S=S_\Sigma$ be a projective toric surface. 
For a maximal cone $\sigma=\mathbb{R}_{\ge 0}x_1+
\mathbb{R}_{\ge 0}x_2\in\Sigma$ and 
a primitive lattice point $y$ in the interior of $\sigma$, 
we put 
\[
\widetilde{\Sigma}:=\left(\Sigma\setminus\{\sigma\}\right)
\cup\left\{\mathbb{R}_{\ge 0}y,\ \mathbb{R}_{\ge 0}x_1+
\mathbb{R}_{\ge 0}y,\ \mathbb{R}_{\ge 0}y+
\mathbb{R}_{\ge 0}x_2\right\}. 
\]
Namely, $\widetilde{\Sigma}$ is obtained from $\Sigma$ by 
subdividing $\sigma$ at $\mathbb{R}_{\ge 0}y$. Thus, 
we obtain the projective toric surface $\widetilde{S}:=
S_{\widetilde{\Sigma}}$ of Picard number 
$\rho(\widetilde{S})=\rho(S)+1$. 

Without loss of generality, we may assume 
\[
x_1=(1,0),\ x_2=(p,q),\ y=(r,s),
\]
where $p,q,r,s\in\mathbb{Z}$ and $q>0,s>0,qr-ps>0$. 
The linear relation 
\[
qy=(qr-ps)x_1+sx_2
\]
holds. 
We prepare two additional elements 
$x_3=(a_3,b_3),x_4=(a_4,b_4)
\in\G(\Sigma)\subset\G(\widetilde{\Sigma})$ 
such that 
$\mathbb{R}_{\ge 0}x_1+\mathbb{R}_{\ge 0}x_3$ and 
$\mathbb{R}_{\ge 0}x_2+\mathbb{R}_{\ge 0}x_4$ are 
maximal cones in $\Sigma$ (the case where $x_3=x_4$ 
is admitted). 
Let $D_1,D_2,D_3,D_4$ 
be the torus invariant prime divisors on $S$ associated to 
$x_1,x_2,x_3,x_4$, respectively, while 
$\widetilde{D}_1,\widetilde{D}_2,\widetilde{D}_3,
\widetilde{D}_4,E$ 
be the torus invariant prime divisors on 
$\widetilde{S}$ associated to $x_1,x_2,x_3,x_4,y$, 
respectively. 

\begin{prop}\label{surfacebu}
Under the above setting, we have 
\[
I_{\widetilde{S}/\widetilde{S}}=I_{S/S}
-\frac{1}{qs(qr-ps)}\left((qr-ps)X_1+sX_2-qY\right)^2.
\]
\end{prop}
\begin{proof}
We calculate the intersection numbers of 
$\widetilde{D}_1,\widetilde{D}_2,\widetilde{D}_3,
\widetilde{D}_4,E$. 

On $S$, we have 
\[
D_1D_2=\frac{1}{\left|\det
\left(
\begin{array}{rr}
1 & p \\
0 & q \\
\end{array}
\right)
\right|}
=\frac{1}{q}\
\]
by Proposition \ref{kotensu}. The relations 
\[
D_1+pD_2+a_3D_3+a_4D_4+\cdots=0\mbox{ and }
qD_2+b_3D_3+b_4D_4+\cdots=0
\]
tells us that 
\[
D_1^2=-\frac{p}{q}-a_3D_1D_3\mbox{ and }
D_2^2=-\frac{b_4}{q}D_2D_4.
\]

On the other hand, on $\widetilde{S}$, 
\[
\widetilde{D}_1E=\frac{1}{\left|\det\left(
\begin{array}{rr}
1 & r \\
0 & s \\
\end{array}
\right)\right|}
=\frac{1}{s}\mbox{ and }
\widetilde{D}_2E=\frac{1}{\left|\det\left(
\begin{array}{rr}
r & p \\
s & q \\
\end{array}
\right)\right|}
=\frac{1}{qr-ps}
\]
hold by Proposition \ref{kotensu}. 
The relations 
\[
\widetilde{D}_1+p\widetilde{D}_2+a_3\widetilde{D}_3+
a_4\widetilde{D}_4+rE+\cdots=0\mbox{ and }
q\widetilde{D}_2+b_3\widetilde{D}_3+b_4\widetilde{D}_4+
sE+\cdots=0
\]
tells us that 
\[
\widetilde{D}_1^2=-a_3\widetilde{D}_1\widetilde{D}_3
-\frac{r}{s},\ 
\widetilde{D}_2^2=-\frac{b_4}{q}\widetilde{D}_2\widetilde{D}_4
-\frac{s}{q(qr-ps)}\mbox{ and }
E^2=-\frac{q}{s(qr-ps)}.
\]
Additionally, we remark that 
$D_1D_3=\widetilde{D}_1\widetilde{D}_3$
 and 
 $D_2D_4=\widetilde{D}_2\widetilde{D}_4$. 

Thus, we obtain 
\[
I_{\widetilde{S}/\widetilde{S}}-I_{S/S}=
-\frac{2}{q}X_1X_2+\left(\frac{p}{q}-\frac{r}{s}\right)X_1^2-\frac{s}{q(qr-ps)}X_2^2
\]
\[
+\frac{2}{s}X_1Y+\frac{2}{qr-ps}X_2Y-\frac{q}{s(qr-ps)}Y^2,
\]
where $X_1,X_{2},X_3,X_4,Y$ are the 
independent variables of polynomials corresponding 
to $x_1,x_2,x_3,x_4,y$, respectively. 
\end{proof}

The following is an immediate consequence of Proposition \ref{surfacebu}:

\begin{cor}\label{surfacegamma}
Let  $\varphi:\widetilde{S}
\to S$ be a toric birational morphism 
between projective toric surfaces $\widetilde{S}$ 
and $S$. Then, $\gamma_2(\widetilde{S})<\gamma_2(S)$. 
\end{cor}

By Corollary \ref{surfacegamma}, for the classification, 
it is sufficient to 
investigate Gorenstein toric del Pezzo surfaces of lower Picard 
numbers. The Gorenstein toric del Pezzo surfaces are 
classified by Koelman \cite{koelman}. There exist 
exactly $16$ Gorenstein toric del Pezzo surfaces, and 
we can get the list of them on the internet (see 
\cite{fano3fold}). 

The case of Picard number $1$ is done by 
Proposition \ref{rho1general}. For the next case where 
the Picard number is two, we show the following easy 
lemma:

\begin{lem}\label{symmetricpair}
Let $S=S_\Sigma$ be a projective toric surface of $\rho(S)=2$. 
If there exist $x_1,x_2\in\G(\Sigma)$ such that $x_1+x_2=0$, 
then $S$ is $\gamma_2$-nef. 
\end{lem}
\begin{proof}
Let $\G(\Sigma)=\{x_1,x_2,x_3,x_4\}$. Then, we may assume 
that there exists another relation $ax_3+bx_4-cx_1=0$ for 
$a,b,c\in\mathbb{Z}_{\ge 0}\ (a,b>0)$. By Proposition 
\ref{rho2lemma},  
there exists a positive rational number $\alpha$ such that 
$\alpha \gamma_2=ac\times(1+1)+2a\times(-c)=0$.
\end{proof}

There exists exactly one Gorenstein $\gamma_2$-nef 
(but not $\gamma_2$-positive) 
toric del Pezzo surface $S=S_\Sigma$ of $\rho(S)=2$ such that 
there is no centrally symmetric pair in $\G(\Sigma)$ 
(see ID $9$ in Table $1$ below). Therefore, 
by Corollary \ref{surfacegamma}, 
if $\rho(S)\ge 3$, then $S$ is not $\gamma_2$-nef. 

Thus, we obtain the following list of 
Gorenstein $\gamma_2$-nef 
toric del Pezzo surfaces. 
In the first column of the table, we describe 
ID of Kasprzyk's classification list: 
\[
\mathtt{
http://www.grdb.co.uk/forms/toricldp}
\] 
We describe a toric surface $S_\Sigma$ by giving 
all the elements in 
$\G(\Sigma)$ in the second column. 

\newpage

\begin{table}[htb]
\begin{center}
\caption{Gorenstein $\gamma_2$-nef 
toric del Pezzo surfaces}
\begin{tabular}{|c|l|}
\hline
ID  & $\G(\Sigma)$   \\
\hline\hline

$12$  & 
$(0,1), (3,1), (-3,-2)$ 
\\
\hline
$13$  & 
$(0,1), (4,1), (-2,-1)$ 
\\
\hline
$14$  & 
$(0,1), (2,1), (-3,-2)$ 
\\
\hline
$15$  & 
$(0,1), (2,1), (-1,-1)$ 
\\
\hline
$16$  & 
$(0,1), (1,1), (-1,-2)$ 
\\
\hline
$6$  & 
$(0,1), (2,1), (0,-1), (-2,-1)$ 
\\
\hline
$8$  & 
$(0,1), (1,1), (0,-1), (-2,-1)$ 
\\
\hline
$9$  & 
$(0,1), (1,1), (-1,-2), (-1,0)$ 
\\
\hline
$10$  & 
$(0,1), (1,1), (0,-1), (-1,0)$ 
\\
\hline
$11$  & 
$(0,1), (1,1), (0,-1), (-1,-1)$ 
\\
\hline
\end{tabular}
\end{center}
\end{table}

\section{Terminal Fano $3$-folds}\label{fano3} %%%%%%%%%%

In this final section, we give the classification of 
$\mathbb{Q}$-factorial terminal 
$\gamma_2$-nef toric 
Fano $3$-folds. 

\begin{defn}
A $\mathbb{Q}$-factorial projective toric 
variety $X$ is a {\em Fano variety} if its anti-canonical 
divisor $-K_X$ is an ample divisor. 
\end{defn}

\begin{rem}
A $\gamma_1$-positive toric variety is nothing but a toric 
Fano variety.
\end{rem}

Terminal toric Fano $3$-folds are classified by 
Kasprzyk \cite{fano3fold}, and the classification list is 
available on the internet (see \cite{fano3fold}). 
There exist exactly $233$ $\mathbb{Q}$-factorial terminal 
toric 
Fano $3$-folds up to isomorphism. 
We checked the non-negativities of $\gamma_2$ for these 
$233$ $\mathbb{Q}$-factorial terminal 
toric Fano $3$-folds mainly by hand calculation. 
Partially, we used the software $\mathtt{polymake}$ (see 
\cite{polymake}). 
Thus, the following is the main theorem of this paper: 
\begin{thm}\label{mfo}
There exist exactly 
$23$ $\mathbb{Q}$-factorial terminal 
$\gamma_2$-nef toric 
Fano $3$-folds. 
\end{thm}

The following is the classification table. In the first column of the table, we describe 
ID of Kasprzyk's classification list: 
\[
\mathtt{
http://www.grdb.co.uk/forms/toricf3t
}
\]
We describe a toric variety $X_\Sigma$ by giving 
all the elements in 
$\G(\Sigma)$ in the second column. 
In the 3rd column, if there exists a Fano contraction 
$\varphi:X\to S$ such that $S$ is 
a Gorenstein $\gamma_2$-nef toric del Pezzo surface, 
then we describe ID of $S$ 
in Table $1$ in Section \ref{delpezzo}.   

\newpage

\begin{table}[htb]
\begin{center}
\caption{$\mathbb{Q}$-factorial terminal 
$\gamma_2$-nef toric 
Fano $3$-folds}
\begin{tabular}{|c|l|l|}
\hline
ID  & $\G(\Sigma)$  & $S$ \\
\hline\hline

$1$  & 
$(1,0,0), (0,1,0), (1,-3,5), (-2,2,-5)$ &
\\
\hline
$2$  & 
$(1,0,0), (0,1,0), (1,-2,7), (-1,1,-5)$ &
\\
\hline
$3$  & 
$(1,0,0), (0,1,0), (-2,2,7), (1,-2,-5)$ &
\\
\hline
$4$  & 
$(1,0,0), (0,1,0), (0,0,1), (-1,-1,-1)$ &
\\
\hline
$5$  & 
$(1,0,0), (0,1,0), (1,-2,5), (-1,1,-4)$ &
\\
\hline
$6$  & 
$(1,0,0), (0,1,0), (-2,1,5), (1,-1,-3)$ &
\\
\hline
$7$  & 
$(1,0,0), (0,1,0), (1,1,2), (-1,-1,-1)$ &
\\
\hline
$8$  & 
$(1,0,0), (0,1,0), (1,-2,3), (-1,1,-2)$ &
\\
\hline
$9$  & 
$(1,0,0), (0,1,0), (1,2,3), (-1,-1,0), (-1,-2,-3)$ & $12$
\\
\hline
$24$  & 
$(1,0,0), (0,1,0), (0,0,1), (-1,-1,0), (0,0,-1)$ & $16$
\\
\hline
$34$  & 
$(1,0,0), (0,1,0), (-2,1,5), (1,-1,-3), (-1,1,3)$ & $14$
\\
\hline
$35$  & 
$(1,0,0), (0,1,0), (1,1,2), (-1,-1,-1), (1,1,1)$ & $16$
\\
\hline
$36$  & 
$(1,0,0), (0,1,0), (0,0,1), (-1,-1,-1), (-1,0,0)$ & $16$
\\
\hline
$38$  & 
$(1,0,0), (0,1,0), (1,-2,3), (-1,1,-2), (-1,0,0)$ & $14$
\\
\hline
$43$  & 
$(1,0,0), (0,1,0), (1,-2,3), (-1,1,-2), (1,-1,2)$ & $15$
\\
\hline
$45$  & 
$(1,0,0), (0,1,0), (1,1,2), (-1,-1,-1), (-1,0,0)$ & $15$
\\
\hline
$47$  & 
 $(1,0,0), (0,1,0), (1,1,2), (-1,0,0), (0,-1,0), (-1,-1,-2)$ 
 & $6,6,6$
\\
\hline
$62$  & 
 $(1,0,0), (0,1,0), (0,0,1), (-1,0,0), (0,-1,0), (0,0,-1)$ 
& $11,11,11$
\\
\hline
$105$  & 
 $(1,0,0), (0,1,0), (1,1,1), (-1,-1,0), (0,0,-1), (1,1,0)$
 & $11$
\\
\hline
$110$  & 
 $(1,0,0), (0,1,0), (1,1,2), (-1,-1,-1), (-1,0,0), (0,-1,0)$
 & $8,8$
\\
\hline
%$111$  & 
% $(1,0,0), (0,1,0), (1,1,2), (-1,-1,-1), (-1,0,0), (-2,-2,-3)$
%\\
%\hline
$123$  & 
 $(1,0,0), (0,1,0), (0,0,1), (-1,-1,0), (0,0,-1), (-1,0,0)$
 & $10,11$
\\
\hline
$131$  & 
 $(1,0,0), (0,1,0), (0,0,1), (-1,-1,-1), (-1,0,0), (-1,-1,0)$
 & $10$
\\
\hline
$140$  & 
 $(1,0,0), (0,1,0), (0,0,1), (-1,-1,-1), (-1,0,0), (-1,-1,-2)$
 & $9$
\\
\hline

\end{tabular}
\end{center}
\end{table}

\begin{rem}
In the above table, every $\mathbb{Q}$-factorial terminal 
$\gamma_2$-nef toric 
Fano $3$-fold $X$ of $\rho(X)\ge 2$ is not 
$\gamma_2$-positive. 
\end{rem}

So, we suggest the following question:

\begin{que}\label{question1}
Does there exist a $\mathbb{Q}$-factorial 
terminal projective $\gamma_2$-positive toric 
variety $X$ of $\rho(X)\ge 2$?
\end{que}

The following example tells us 
that the answer to Question \ref{question1} is {\em positive}, 
if $X$ is not terminal:

\begin{ex}
Let $\Sigma$ be a complete fan in $N=\mathbb{Z}^2$ 
such that $\G(\Sigma)=\{x_1:=(1,0),x_2:=(1,2),
x_3:=(-1,2),x_4:=(-1,-1)\}$, and $S=S_\Sigma$ 
the associated projective toric surface. 
One can easily find the two relations 
\[
2x_1+x_3-x_2=0\mbox{ and }3x_2+4x_4-x_3=0.
\]
By Proposition \ref{rho2lemma}, for a positive number $\alpha$, 
we have 
\[
\alpha\gamma_2=3\times(1+2^2+1)+
2\times 3\times(-1-3)+(3^2+4^2+1)=20>0.
\]
Therefore, $S$ is $\gamma_2$-positive. We remark that 
$S$ has two canonical singular points and one log terminal 
singular point.
\end{ex} 

On the other hand, for every $\gamma_2$-nef toric 
Fano $3$-fold $X$ of $\rho(X)\ge 2$ in Table $2$, 
there exists a Fano contracton $\varphi:X\to S$ 
such that $\dim S=2$ and $S$ is $\gamma_2$-nef. 
So, the following is also a natural question. 

\begin{que}\label{question2}
For any $\mathbb{Q}$-factorial terminal projective 
$\gamma_2$-nef toric 
$d$-fold of $\rho(X)\ge 2$, does one of the following 
hold? 
\begin{enumerate}
\item There exists a Fano contraction $\varphi:X\to \overline{X}$ 
such that $\overline{X}$ is a %$\mathbb{Q}$-factorial 
$\gamma_2$-nef toric $(d-1)$-fold. 
\item There exists a toric finite morphism 
$\pi:X'\to X$ such that $X'$ is a direct product of lower-dimensional  
%$\mathbb{Q}$-factorial 
$\gamma_2$-nef toric varieties (see Proposition \ref{directproduct}).
\end{enumerate}
\end{que}

\begin{rem}
Without the assumption that $X$ is terminal, 
the Gorenstein $\gamma_2$-nef 
toric del Pezzo surface of ID $9$ in Table $1$ in 
Section \ref{delpezzo} 
tells us that the answer to 
Question \ref{question2} 
is negative. 
\end{rem}

\begin{rem}
There exists a $\mathbb{Q}$-factorial terminal toric 
$3$-fold such that it has a Fano contraction to a 
$\gamma_2$-nef surface, 
but is not $\gamma_2$-nef. For example, let $X=X_\Sigma$ be 
the smooth toric Fano $3$-fold such that 
\[
\G(\Sigma)=\{(1,0,0), (0,1,0), (0,0,1), (-1,-1,-1), (-1,-1,0), (1,1,0)\}. 
\]
Then, $X$ is a $\mathbb{P}^1$-bundle over 
$\mathbb{P}^1\times\mathbb{P}^1$, but not $\gamma_2$-nef. 
\end{rem}

\smallskip

We end this section by giving an example of 
calculations for a terminal toric Fano $3$-fold.

\begin{ex}
Let $X=X_\Sigma$ be the 
$\mathbb{Q}$-factorial terminal 
toric Fano $3$-fold of ID $34$ in the above table. 
Put 
$v_1:=(1,0,0),\ v_2:=(0,1,0),\ v_3:=(-2,1,5),
\ v_4:=(1,-1,-3),\ v_5:= (-1,1,3)$, and put 
$D_1,\ldots,D_5$ be the torus invariant 
divisors corresponding to $v_1,\ldots,v_5$, respectively. 
Then, we have a $3$ relations
\[
D_1-2D_3+D_4-D_5=0,\ D_2+D_3-D_4+D_5=0,\ 
5D_3-3D_4+3D_5=0
\]
in $\N^1(X)$. There exist exactly $6$ maximal cones 
generated by 
\[
\{v_1,v_2,v_4\},\ \{v_2,v_3,v_4\},\ \{v_1,v_3,v_4\},\ 
\{v_1,v_2,v_5\},\ \{v_2,v_3,v_5\},\ \{v_1,v_3,v_5\}.
\]
%So, we have $D_1D_2D_3=0$ and $D_4D_5=0$. 
The equalities $3D_4=5D_3+3D_5$, $3D_1=D_3$ 
and $2D_1=D_2$ 
say that it is sufficient to check the non-negativities 
for two torus invariant surfaces $S_1$ and $S_5$ 
corresponding to $1$-dimensional cones 
$\mathbb{R}_{\ge 0}v_1$ and 
$\mathbb{R}_{\ge 0}v_5$, respectively. 
Since $\rho(S_5)=1$, we have $\gamma_2\cdot S_5>0$ by 
Proposition \ref{rho1lemma}.  
On the other hand, since $\rho(S_1)=2$, we can apply 
Proposition \ref{rho2lemma}. 
One can easily calculate the relations 
\[
2v_2+3v_3-5v_5+v_1=0\mbox{ and }
v_4+v_5=0
\]
corresponding to $2$-dimensional cones 
\[
\mathbb{R}_{\ge 0}v_1+\mathbb{R}_{\ge 0}v_5
\mbox{ and }
\mathbb{R}_{\ge 0}v_1+\mathbb{R}_{\ge 0}v_2,
\]
respectively. By Proposition \ref{rho2lemma}, there 
exists a positive rational number $\alpha$ such that 
\[
\alpha I_{S_1/X}=4(2V_2+3V_2-5V_5+V_1)(V_4+V_5)
+10(V_4+V_5)^2,
\]
where $V_1,\ldots,V_{5}$ are the 
independent variables of polynomials corresponding 
to $v_1,\ldots,v_{5}$, respectively. In particular, 
$\alpha \gamma_2\cdot S_1=4\times(-5)+10+10=0$. 
Therefore, $X$ is $\gamma_2$-nef (not $\gamma_2$-positive). 
\end{ex}

\end{document}